\documentclass[a4paper,12pt]{amsart}
\usepackage[utf8]{inputenc}

\usepackage[T1]{fontenc}
\usepackage{lmodern}
\usepackage{microtype}
\usepackage{amsmath,amssymb,amsthm,mathtools}
\usepackage[margin=30mm]{geometry}
\usepackage[colorlinks=true,citecolor=blue,linkcolor=blue,urlcolor=blue]{hyperref}

\usepackage[backend=bibtex,style=alphabetic,maxbibnames=99,maxcitenames=3]{biblatex}
\AtBeginBibliography{\small\sloppy\setlength{\bibitemsep}{0pt}}
\AtEveryBibitem{\clearfield{doi}\clearfield{eprint}\clearfield{archivePrefix}\clearfield{primaryClass}\clearfield{url}\clearfield{urldate}}

\newtheorem{theorem}{Theorem}[section]
\newtheorem{proposition}[theorem]{Proposition}
\newtheorem{lemma}[theorem]{Lemma}

\theoremstyle{definition}
\newtheorem{definition}[theorem]{Definition}
\theoremstyle{remark}

\newcommand{\OO}{\mathcal O}
\newcommand{\CC}{\mathcal C}
\newcommand{\relint}{\operatorname{relint}}
\newcommand{\dir}{\operatorname{dir}}

\title{The Hibi--Li Face-Number Conjecture}

\author{Ghislain Fourier}
\address{Chair of Algebra and Representation Theory, RWTH Aachen University, Pontdriesch 10--16, 52062 Aachen, Germany}
\email{fourier@art.rwth-aachen.de}



\begin{document}

\begin{abstract}
We prove the face-number conjecture of Hibi and Li: for every finite poset, the order polytope has no more faces of any given dimension than the chain polytope.
More generally, the face numbers increase weakly along the admissible family of chain--order polytopes when one order coordinate is replaced by a chain coordinate. 
We also prove that, unless the order and chain polytopes are unimodularly equivalent, all face-number inequalities in dimensions $2,\ldots,|P|-1$ are strict.

\end{abstract}
\maketitle
\section{Introduction}\label{sec1}

The order polytope $\mathcal O(P)$ and the chain polytope $\mathcal C(P)$ of a finite poset $P$ were studied by Stanley \cite{Sta86}, following Geissinger's work on the faces of the order polytope \cite{Gei81}.
Stanley's transfer map shows that the two polytopes have the same Ehrhart polynomial; their numbers of vertices also agree.
Hibi and Li \cite{HL16} proved that the order polytope has at most as many facets as the chain polytope, with equality precisely when the two polytopes are unimodularly equivalent.
They conjectured the corresponding inequality for faces of every dimension.

Hibi, Li, Sahara and Shikama \cite{Hib+17} proved that the numbers of edges always agree.
Descriptions of edges and low-dimensional faces were obtained in \cite{HL19,Mor19}, and simplex faces were studied further in \cite{Mor25a,Mor25b,FL26}.
Ahmad, the author and Joswig \cite{AFJ26} proved the conjecture for maximal ranked posets while Freij-Hollanti and Lundstr\"om \cite{FL24} proved it for the larger class generated from $X$-free posets by disjoint unions and ordinal sums.
Freij-Hollanti, Lundstr\"om and Mori \cite{FLM26} proved the two-dimensional inequality for arbitrary finite posets, including its equality characterization.
The difference between the numbers of facets was studied in \cite{Bha+25,Fou26}.

Marked order and chain polytopes were introduced by Ardila, Bliem and Salazar \cite{ABS11}; they provided a generalization of Stanley's transfer map.
The author \cite{Fou16} extended the facet comparison to marked posets and formulated the marked face-number conjecture.
Fang and the author \cite{FF16} introduced marked chain--order polytopes and conjectured that face numbers increase when order coordinates are replaced by chain coordinates along admissible decompositions.
These polytopes were later placed in a continuous family \cite{Fan+20}.
The face structure of marked order polyhedra and Minkowski decompositions of marked chain--order polytopes were studied in \cite{Peg18,FFP20}, respectively.

In this paper we prove the face-number inequalities conjectured by Hibi and Li.
For a polytope $Q$, let $f_k(Q)$ denote the number of its $k$-dimensional faces.
\begin{theorem}\label{thm:face-inequality}
Let $P$ be a finite poset with $n$ elements.
Then
\[
 f_k(\mathcal O(P))\leq f_k(\mathcal C(P))  \qquad (0\leq k\leq n-1).
\]
\end{theorem}
Theorem~\ref{thm:marked-strictness} also proves the second part of the Hibi--Li conjecture \cite[Conjecture~2.4(b)]{HL16} in the range $2\leq k\leq n-1$: equality in any one of these dimensions holds if and only if $\OO(P)$ and $\CC(P)$ are unimodularly equivalent.
Otherwise all these inequalities are strict.
The restriction $k\geq2$ is necessary, since the numbers of vertices and edges always agree.
The theorem is proved more generally for marked polytopes.
The equality question is also discussed in the recent preprint \cite[Section~5]{Lee26}.

We also prove monotonicity along the admissible chain--order polytopes.
For an order ideal $I\subseteq P$, let $\OO_{I,P\setminus I}(P)$ be the polytope with chain coordinates indexed by $I$ and order coordinates indexed by $P\setminus I$, as recalled in Section~\ref{sec2}.
Thus $\OO_{\varnothing,P}(P)=\mathcal O(P)$ and $\OO_{P,\varnothing}(P)=\mathcal C(P)$.
We show that
\[
 f_k(\OO_{I,P\setminus I}(P))\leq f_k(\OO_{I',P\setminus I'}(P))  \qquad  (I\subseteq I',\quad 0\leq k\leq n-1)
\]
for all order ideals $I,I'$ of $P$.
This proves the admissible, unmarked case of \cite[Conjecture~5.3]{FF16}.

\medskip
\noindent\textbf{Outline of the proof.}
We pass through the admissible chain--order polytopes of Section~\ref{sec2}, replacing one order coordinate by a chain coordinate at each step.
Let $Q$ and $Q'$ be two consecutive polytopes, and let $q$ be the changing coordinate.
Section~\ref{sec3} shows that projection away from $q$ gives a common base $B$, with fibres
$$
 [L(u),U(u)] \quad\text{and}\quad [0,U(u)-L(u)]
$$
for $Q$ and $Q'$, respectively.
Here $L(u)$ is the maximum at $u$ of the chain sums below $q$, and $U(u)$ is the minimum at $u$ of the order coordinates above $q$ and the constant $1$.
The transfer map subtracts $L(u)$ from the changing coordinate.
It is piecewise affine and need not send faces to faces.

Let $\mathcal S$ and $\mathcal T$ be the subdivisions of $B$ into domains of affinity of $L$ and $U$.
The lower and upper graph faces of $Q$ are indexed by $\mathcal S$ and $\mathcal T$.
For $Q'$, the lower graph is $B$, and the upper graph has subdivision $\mathcal S\vee\mathcal T$.
Write $F_{\mathcal K}(t)=\sum_{\varnothing\ne A\in\mathcal K}t^{\dim A}$ for the cell polynomial of a polytopal complex, using the full face complex for a polytope.
Accounting also for faces lifted from $B$ and for degenerate fibres, Section~\ref{sec4} proves
$$
 F_{Q'}-F_Q  = F_B+F_{\mathcal S\vee\mathcal T}-F_{\mathcal S}-F_{\mathcal T}.
$$
Under relative transversality, the right-hand side is coefficientwise nonnegative.
Section~\ref{sec4} proves this by counting cells with the same carrier, the smallest face of $B$ containing them.

Section~\ref{sec5} establishes relative transversality.
At a point $u$ with $L(u)<U(u)$, we clip the corresponding order coordinates at a level $h$ between these bounds.
The partial transfer map respects this decomposition, and the normalized summands lie in the same face of $B$ as $u$.
This separates variations in the coordinates below and above $q$ and shows that the conormal spaces of the two subdivision cells have zero intersection.
Where the fibre degenerates, both restricted subdivisions are trivial.
Section~\ref{sec6} combines these statements and iterates the one-coordinate inequality along a linear extension of $P$.
Section~\ref{sec7} gives the face comparison explicitly for the five-element star.
Section~\ref{sec:extensions} proves marked monotonicity and the equality case of the Hibi--Li conjecture.
The results of Section~\ref{sec:extensions} were obtained before the author received Goel's manuscript~\cite{Goe26}.
That manuscript prompted their inclusion in this arXiv update.

\medskip
\noindent\textbf{Acknowledgements.}
The author gratefully acknowledges support by the Deutsche Forschungsgemeinschaft (DFG, German Research Foundation) -- Project-ID 286237555 -- TRR 195.
AI tools were used for language editing, bibliographic searches, and auxiliary checks of the final manuscript.

%%%%%%%%%%%%%%%%%%%%%%%%%%%%%%%%%%%%%%%%%%%%%%%%%%%%%%%%%%%%%%%%%
%%%%%%%%%%%%%%%%%%%%%%%%%%%%%%%%%%%%%%%%%%%%%%%%%%%%%%%%%%%%%%%%%

\section{Intermediate chain--order polytopes}\label{sec2}

Throughout the paper, $P$ is a finite poset.
We allow the empty chain and take its sum to be zero.
We set
\[
 \OO(P)=\{v\in[0,1]^P:v_p\leq v_q\text{ whenever }p\leq q\}
\]
and
\[
 \CC(P)=\left\{x\in\mathbb R_{\geq0}^P: \sum_{p\in K}x_p\leq1\text{ for every chain }K\subseteq P\right\}.
\]

Let $I\subseteq P$ be an order ideal and put $J=P\setminus I$.
Define $\OO_{I,J}(P)$ by the inequalities
\begin{align*}
 x_i&\geq0 &&(i\in I),\\
 0\leq x_j&\leq1 &&(j\in J),\\
 x_j&\leq x_{j'} &&(j\leq j'\text{ in }J),\\
 \sum_{i\in K}x_i&\leq x_j    &&\left(\substack{j\in J,\ K\subseteq I\text{ a chain},\\
 i<j\text{ for every }i\in K}\right),\\
 \sum_{i\in K}x_i&\leq1 &&(K\subseteq I\text{ a chain}).
\end{align*}
In particular, $\OO_{\varnothing,P}(P)=\OO(P)$ and $\OO_{P,\varnothing}(P)=\CC(P)$.
These are marked chain--order polytopes of \cite{FF16,Fan+20} obtained by adjoining a minimum marked $0$ and a maximum marked $1$.

\begin{lemma}\label{lem:transfer}
The map $T_I:\OO(P)\to\OO_{I,J}(P)$ given by
\[
 (T_Iv)_i=v_i-\max\bigl(\{v_p:p<i\}\cup\{0\}\bigr)\quad(i\in I),  \qquad  (T_Iv)_j=v_j\quad(j\in J)
\]
is a piecewise-linear bijection.
Its inverse is
\[
 v_i=\max\left\{\sum_{p\in K}x_p: K\subseteq I\text{ is a chain ending at }i\right\}, \qquad v_j=x_j\quad(j\in J).
\]
The same assertion holds with upper bound $1$ replaced by any $h\geq0$; the image is then $h\OO_{I,J}(P)$.
\end{lemma}

This is a special case of \cite[Theorem~2.1]{Fan+20}.
We recall the proof to fix the inverse map.

\begin{proof}
Since $I$ is an order ideal, all predecessors of an element of $I$ belong to $I$.
For $x\in\OO_{I,J}(P)$ define $v$ successively along a linear extension by
\[
 v_i=x_i+\max\bigl(\{v_p:p<i\}\cup\{0\}\bigr)\quad(i\in I),  \qquad v_j=x_j\quad(j\in J).
\]
A chain ending at $i$ is either $\{i\}$ or a chain ending at a predecessor followed by $i$.
Hence the recurrence gives the stated maximum over chain sums, and $(T_Iv)_i=x_i$.

Nonnegativity of $x_i$ implies $v_p\leq v_i$ for $p<i$ in $I$.
The chain bounds give $v_i\leq1$.
If $i\in I$, $j\in J$, and $i<j$, every chain used to compute $v_i$ lies below $j$, so $v_i\leq x_j=v_j$.
The order inequalities inside $J$ hold by definition.
Since $I$ is an ideal, no element of $J$ lies below an element of $I$.
Thus $v\in\OO(P)$.

Conversely, let $v\in\OO(P)$ and set $x=T_Iv$.
Then $x_i\geq0$ for $i\in I$, and for any chain $i_1<\cdots<i_r$ in $I$ we have
\[
 \sum_{s=1}^r x_{i_s}  \leq v_{i_1}+\sum_{s=2}^r(v_{i_s}-v_{i_{s-1}})  =v_{i_r}.
\]
If the chain lies below $j\in J$, this is at most $v_j=x_j$.
In all cases it is at most $1$.
Hence $x\in\OO_{I,J}(P)$, and the recurrence recovers $v$.
The case $h>0$ follows by scaling; for $h=0$ both sets consist of the zero vector.
\end{proof}

The inequalities defining $\OO_{I,J}(P)$ are finite in number, and the singleton-chain inequalities bound each chain coordinate.
Thus $\OO_{I,J}(P)$ is a polytope.
It has dimension $|P|$: choose a strictly order-preserving $v$ with pairwise distinct coordinates in $(0,1)$.
The estimate above shows that $T_Iv$ satisfies all nontrivial defining inequalities strictly.
For $P=\varnothing$, the polytope is a point.

\section{A one-element change}\label{sec3}

Let $q$ be minimal in $J=P\setminus I$, put $P'=P\setminus\{q\}$, and set
\[
 B=\OO_{I,J\setminus\{q\}}(P').
\]
We give $P'$ the induced order.
Deleting the $q$-coordinate of a point of $\OO_{I,J}(P)$ leaves all inequalities defining $B$ satisfied.
Thus projection along this coordinate has image contained in $B$.
Conversely, for $u\in B$, let $v=T_I^{-1}u\in\OO(P')$.
We add a $q$-coordinate to the vector $(v_p)_{p\in P'}$, choosing
\[
 v_q\in\left[   \max\bigl(\{v_p:p<q\}\cup\{0\}\bigr),   \min\bigl(\{v_r:r>q\}\cup\{1\}\bigr)\right].
\]
Every predecessor of $q$ is below every successor in the induced order on $P'$, so the interval is nonempty.
The resulting vector $(v_p)_{p\in P}$ belongs to $\OO(P)$.
Moreover, no element of $I$ lies above $q$, so applying $T_I$ to this vector leaves all coordinates indexed by $P'$ equal to those of $u$.
Thus $u$ is the projection of a point of $\OO_{I,J}(P)$, proving surjectivity.
All predecessors of $q$ lie in $I$, since $q$ is minimal in $J$, and all successors lie in $J$, since $I$ is an ideal.
In particular, $I\cup\{q\}$ is again an order ideal.

For $u\in B$ define
\begin{align*}
 L(u)&=\max\left( \left\{\sum_{p\in K}u_p: K\subseteq I\text{ is a chain and }p<q\text{ for every }p\in K\right\} \cup\{0\}\right),\\
 U(u)&=\min\bigl(\{u_r:r>q\}\cup\{1\}\bigr).
\end{align*}
The function $L$ is convex piecewise affine and depends only on coordinates below $q$.
The function $U$ is concave piecewise affine and depends only on the order coordinates above $q$.
If $K$ occurs in the definition of $L$ and $r>q$, then the defining inequalities of $B$ give $\sum_{p\in K}u_p\leq u_r$ and $\sum_{p\in K}u_p\leq1$.
It follows that
\[
 0\leq L(u)\leq U(u)\leq1 \qquad (u\in B).
\]

\begin{proposition}\label{prop:fibers}
With coordinates $(u,z)$ and $(u,w)$ over $B$, respectively,
\begin{align}
 \OO_{I,J}(P)&=\{(u,z):u\in B,\ L(u)\leq z\leq U(u)\},\label{eq:oldfiber}\\
 \OO_{I\cup\{q\},J\setminus\{q\}}(P)&= \{(u,w):u\in B,\ 0\leq w\leq U(u)-L(u)\}.\label{eq:newfiber}
\end{align}
\end{proposition}

\begin{proof}
For the order coordinate $z=x_q$, nonnegativity and the chain inequalities below $q$ give $z\geq L(u)$.
The inequalities $z\leq1$ and $z\leq u_r$ for $r>q$ give $z\leq U(u)$.
The inequalities not involving $z$ are precisely those of $B$, since $P'$ has the induced order.
This proves \eqref{eq:oldfiber}.

No element of $I$ lies above $q$, so $q$ is maximal in $I\cup\{q\}$.
Each chain containing $q$ therefore consists of a chain $K\subseteq I$ below $q$, followed by $q$.
For the chain coordinate $w=x_q$, the corresponding inequalities are
\[
 w+\sum_{p\in K}u_p\leq u_r\quad(r>q)
 \qquad\text{and}\qquad
 w+\sum_{p\in K}u_p\leq1,
\]
together with $w\geq0$.
Maximizing over $K$ and minimizing the upper bounds gives \eqref{eq:newfiber}.
Here the chain inequalities not involving $q$, together with the order inequalities on $J\setminus\{q\}$, are precisely the inequalities defining $B$.
\end{proof}

The change from \eqref{eq:oldfiber} to \eqref{eq:newfiber} is the piecewise-affine bijection
\[
 \tau_q(u,z)=(u,z-L(u)),\qquad
 \tau_q^{-1}(u,w)=(u,w+L(u)).
\]
It equals $T_{I\cup\{q\}}\circ T_I^{-1}$, where both partial transfer maps are taken on $P$.
Thus only the coordinate indexed by $q$ changes.
The lower graph is sent to $B\times\{0\}$ and the upper graph to the graph of $U-L$.
The map need not be globally affine.

\section{Subdivisions and face counting}\label{sec4}

For a finite polytopal complex $\mathcal K$, including all nonempty faces, set
\[
 F_{\mathcal K}(t)=  \sum_{A\in\mathcal K,\ A\neq\varnothing}t^{\dim A}.
\]
For a polytope, we use its full face complex, including the polytope itself.

Let $\mathcal S$ and $\mathcal T$ be the subdivisions of $B$ into maximal domains on which $L$ and $U$ are affine, together with their faces, and let $\mathcal S\vee\mathcal T$ be their common refinement.
If
\[
 L=\max_i\ell_i,
\]
then a maximal cell of $\mathcal S$ is of the form
\[
 A_i=\{u\in B:L(u)=\ell_i(u)\}
\]
with nonempty interior relative to $\operatorname{aff}(B)$.
For every face $E$ of $B$, the nonempty intersections $A_i\cap E$, together with their faces, form a subdivision of $E$.
We denote this restricted subdivision by $\mathcal S|_E$.
We use the analogous notation for $\mathcal T$ and $\mathcal S\vee\mathcal T$.

The \emph{carrier} of a cell is the smallest face of $B$ containing it.
If its carrier is $E$, then its relative interior is contained in $\relint(E)$: a supporting hyperplane of $B$ containing a relative interior point of the cell contains the whole cell.

Here the cells $A\in\mathcal S$, $C\in\mathcal T$, and their carriers all lie in the base $B$.
Their corresponding graph pieces in the old fibre polytope are
\[
 \{(u,L(u)):u\in A\},\qquad  \{(u,U(u)):u\in C\}.
\]
Projection $(u,z)\mapsto u$ is an affine isomorphism on each such piece.
Lemma~\ref{lem:faceidentity} identifies these pieces with the lower and upper graph faces, respectively.

\begin{lemma}\label{lem:faceidentity}
One has
\begin{equation}\label{eq:faceidentity}
 F_{\OO_{I\cup\{q\},J\setminus\{q\}}(P)}-F_{\OO_{I,J}(P)}  =F_B+F_{\mathcal S\vee\mathcal T}-F_{\mathcal S}-F_{\mathcal T}.
\end{equation}
\end{lemma}

\begin{proof}
Put
\[
 H=U-L
\]
and
\[
 D=\{u\in B:H(u)=0\}.
\]
By Section~\ref{sec3}, $H\geq0$ on $B$, and $H$ is concave.

\smallskip
\noindent\emph{Step 1: degenerate fibres.}
We first show that $D$ is a union of faces of $B$.
Suppose that $x\in\relint(E)$ for a face $E$ of $B$ and $H(x)=0$.
For $y\in E$, choose $\varepsilon>0$ such that
\[
 y'=x+\varepsilon(x-y)\in E.
\]
Writing
\[
 x=\frac{\varepsilon}{1+\varepsilon}y+\frac{1}{1+\varepsilon}y',
\]
concavity and nonnegativity of $H$ give
\[
 0=H(x)\geq \frac{\varepsilon}{1+\varepsilon}H(y)+\frac{1}{1+\varepsilon}H(y')\geq0.
\]
Thus $H(y)=0$, and hence $E\subseteq D$.

If $E\subseteq D$, then $L|_E=U|_E$.
Since $L|_E$ is convex and $U|_E$ is concave, their common restriction is affine.
Thus the restrictions of $\mathcal S$ and $\mathcal T$ to $E$ are trivial.
Consequently, the cells of either subdivision contained in $D$ are precisely the faces of $B$ contained in $D$.
Let $F_D$ denote the face polynomial of this subcomplex.

\smallskip
\noindent\emph{Step 2: classification of faces.}
Consider first
\[
 \OO_{I,J}(P)=\{(u,z):u\in B,\ L(u)\leq z\leq U(u)\}.
\]
For a maximal affinity domain $A_i=\{u\in B:L(u)=\ell_i(u)\}$ of $L$, the set
\[
 G_i^-=\{(u,z)\in\OO_{I,J}(P):z=\ell_i(u)\}
\]
is a face, since $z\geq L(u)\geq\ell_i(u)$ throughout the polytope.
Projection to $B$ restricts to an affine isomorphism $G_i^-\to A_i$.

Conversely, let $F$ be a nonempty face contained in the lower graph and choose $\xi\in\relint(F)$.
Some $G_i^-$ contains $\xi$.
A face containing a relative interior point of $F$ contains all of $F$, so $F$ is a face of $G_i^-$.
It follows that the lower graph faces are indexed, with dimensions preserved, by the cells of $\mathcal S$.
The same argument applies to the upper graph and $\mathcal T$.
The two collections overlap precisely in the lifted faces of $B$ contained in $D$, and therefore contribute
\[
 F_{\mathcal S}+F_{\mathcal T}-F_D.
\]
The remaining faces are the full inverse images of faces of $B$.
To see this, expose a face by a linear form
\[
 c(u)+\alpha z.
\]
If $\alpha<0$, maximization over each fibre forces $z=L(u)$; if $\alpha>0$, it forces $z=U(u)$.
Thus these faces have already been counted on the two graphs.
If $\alpha=0$, the exposed face is
\[
 \{(u,z):u\in E,\ L(u)\leq z\leq U(u)\}
\]
for a face $E$ of $B$.
For $E\subseteq D$ this has already been counted.
Otherwise $H>0$ on $\relint(E)$, and its dimension is $\dim E+1$.
Hence
\begin{equation}\label{eq:old-face-count}
 F_{\OO_{I,J}(P)}  =F_{\mathcal S}+F_{\mathcal T}-F_D+t(F_B-F_D).
\end{equation}

For
\[
 \OO_{I\cup\{q\},J\setminus\{q\}}(P)  =\{(u,w):u\in B,\ 0\leq w\leq H(u)\},
\]
the same classification applies.
The lower graph is $B\times\{0\}$, while the upper graph is determined by the affinity subdivision $\mathcal R$ of $H$.
Thus
\begin{equation}\label{eq:new-face-count}
 F_{\OO_{I\cup\{q\},J\setminus\{q\}}(P)}  =F_B+F_{\mathcal R}-F_D+t(F_B-F_D).
\end{equation}

\smallskip
\noindent\emph{Step 3: the upper subdivision.}
We claim that
\[
 \mathcal R=\mathcal S\vee\mathcal T.
\]
Let $u,v\in B$, $0<\lambda<1$, and $x=\lambda u+(1-\lambda)v$.
By concavity of $U$ and convexity of $L$,
\[
 U(x)-\lambda U(u)-(1-\lambda)U(v)\geq0
\]
and
\[
 \lambda L(u)+(1-\lambda)L(v)-L(x)\geq0.
\]
Their sum is
\[
 H(x)-\lambda H(u)-(1-\lambda)H(v).
\]
If $H$ is affine on a convex set, this sum vanishes for all $u,v$ in that set, and hence both nonnegative terms vanish.
Thus both $L$ and $U$ are affine there.
The converse is immediate from $H=U-L$.
Therefore the affinity subdivision of $H$ is the common refinement $\mathcal S\vee\mathcal T$.

Substituting
\[
 F_{\mathcal R}=F_{\mathcal S\vee\mathcal T}
\]
in \eqref{eq:new-face-count} and subtracting \eqref{eq:old-face-count} gives \eqref{eq:faceidentity}.
\end{proof}

\begin{definition}\label{def:relative-transverse}
Two subdivisions $\mathcal S$ and $\mathcal T$ of $B$ are \emph{transverse relative to the faces of $B$} if the following holds.
For every face $E$ of $B$ and every $x\in\relint(E)$, let $A$ and $C$ be the unique cells of $\mathcal S|_E$ and $\mathcal T|_E$ whose relative interiors contain $x$.
Then
\begin{equation}\label{eq:transverse}
 \dim(A\cap C)=\dim A+\dim C-\dim E.
\end{equation}
\end{definition}

\begin{lemma}\label{lem:subdivision}
Let $\mathcal S$ and $\mathcal T$ be polytopal subdivisions of a polytope $B$ that are transverse relative to its faces.
Then
\begin{equation}\label{eq:subdivision-inequality}
 F_{\mathcal S}+F_{\mathcal T}\leq F_{\mathcal S\vee\mathcal T}+F_B
\end{equation}
coefficientwise.
\end{lemma}

\begin{proof}
Fix $k$.
We prove the inequality separately for the $k$-cells with a fixed carrier $E$.
Summing these inequalities over all faces $E$ of $B$ proves the inequality.
Put $e=\dim E$.

\noindent\emph{Step 1: carriers of dimension $e\neq k$.}
If $e<k$, there are no $k$-cells with carrier $E$.
Assume $e>k$.
For each $k$-cell $A\in\mathcal S$ with carrier $E$, set
\[
 \mathcal R(A)=  \{R\in\mathcal S\vee\mathcal T:R\subseteq A,\ \dim R=k\}.
\]
This set is nonempty because the common refinement induces a subdivision of $A$.
Moreover, for every $R\in\mathcal R(A)$,
\[
 \relint(R)\subseteq\relint(A)\subseteq\relint(E),
\]
since $R\subseteq A$ and both have dimension $k$.
Thus $R$ also has carrier $E$.
Define $\mathcal R(C)$ analogously for the $k$-cells $C\in\mathcal T$ with carrier $E$.

Within either subdivision, these sets are pairwise disjoint, because the relative interiors of distinct cells are disjoint.
They are also disjoint across the two subdivisions.
Indeed, if $R\in\mathcal R(A)\cap\mathcal R(C)$, then any $x\in\relint(R)$ lies in $\relint(A)\cap\relint(C)\cap\relint(E)$.
Transversality gives
\[
 \dim(A\cap C)=2k-e<k,
\]
contrary to $R\subseteq A\cap C$ and $\dim R=k$.
Hence we have pairwise disjoint, nonempty sets of refinement $k$-cells, one for each $k$-cell of either subdivision with carrier $E$.
This proves the required inequality for $e>k$.

\noindent\emph{Step 2: carriers of dimension $e=k$.}
Now the $k$-cells with carrier $E$ are precisely the full-dimensional cells of the restricted subdivisions $\mathcal S|_E$ and $\mathcal T|_E$.
Form a bipartite graph $G_E$ whose vertex classes are these cells, with an edge  between $A$ and $C$ for every $e$-cell
\[
 R=A\cap C
\]
of $(\mathcal S\vee\mathcal T)|_E$.
Thus the edges of $G_E$ are in bijection with the $e$-cells of the common refinement.

We claim that $G_E$ is connected.
For $e=0$ this is immediate.
Suppose $e>0$ and put
\[
 \mathcal R=(\mathcal S\vee\mathcal T)|_E.
\]
The dual graph of a polytopal subdivision of a polytope is connected; see \cite[Lemma~2]{CJK25}.
Let
\[
 R=A\cap C,\qquad R'=A'\cap C'
\]
be adjacent $e$-cells of $\mathcal R$, and let $F=R\cap R'$ be their common $(e-1)$-face.
We show that either $A=A'$ or $C=C'$.

Suppose that $A\neq A'$ and $C\neq C'$.
Then $A\cap A'$ is an $(e-1)$-cell of $\mathcal S|_E$ and $C\cap C'$ is an $(e-1)$-cell of $\mathcal T|_E$, both containing $F$.
Since $F$ is a common face of two distinct $e$-cells of $\mathcal R$, its relative interior lies in $\relint(E)$.
Moreover, $F$ has the same dimension as $A\cap A'$ and $C\cap C'$, so
\[
 \relint(F)\subseteq\relint(A\cap A')\cap\relint(C\cap C').
\]
Hence, for $x\in\relint(F)$, transversality gives
\[
 \dim\bigl((A\cap A')\cap(C\cap C')\bigr)  =(e-1)+(e-1)-e=e-2.
\]
This is impossible since the intersection contains the $(e-1)$-dimensional face $F$.
Thus $A=A'$ or $C=C'$.

Consequently, adjacent cells in the dual graph of $\mathcal R$ correspond to edges of $G_E$ with a common endpoint.
Since the dual graph of $\mathcal R$ is connected, all edges of $G_E$ lie in one connected component.
Every vertex of $G_E$ is incident with an edge, since a full-dimensional cell of one subdivision cannot be covered by lower-dimensional cells of the other.
Hence $G_E$ is connected.

Therefore
\[
 \#E(G_E)\geq \#V(G_E)-1.
\]
Thus the number of $k$-cells of $\mathcal S\vee\mathcal T$ with carrier $E$ is at least the sum of the corresponding numbers for $\mathcal S$ and $\mathcal T$, minus one.
The missing $1$ is supplied by the face $E$ itself in $F_B$.
\end{proof}

\section{Clipping and relative transversality}\label{sec5}

For $v\in\OO(P')$ and $0<h<1$, define
\[
 v^-_p=\min(v_p,h),  \qquad  v^+_p=\max(v_p-h,0),
\]
and put
\[
 u=T_Iv,  \qquad  u^\pm=T_I(v^\pm).
\]
This clipping decomposition is also used in the proof of \cite[Proposition~15]{FFP20}.

\begin{lemma}\label{lem:clipping}
One has
\begin{equation}\label{eq:clipping}
 u=u^-+u^+,  \qquad  u^-\in hB,  \qquad  u^+\in(1-h)B.
\end{equation}
If $u$ lies in a face $E$ of $B$, then
\begin{equation}\label{eq:same-face}
 u^-/h\in E,  \qquad  u^+/(1-h)\in E.
 \end{equation}
\end{lemma}

\begin{proof}
For $i\in I$, put
\[
 m_i=\max\bigl(\{v_p:p<i\}\cup\{0\}\bigr).
\]
Maximum commutes with both clipping functions, so
\begin{align*}
 (T_Iv^-)_i  &=  \min(v_i,h)-\min(m_i,h),\\
 (T_Iv^+)_i  &=  (v_i-h)_+-(m_i-h)_+.
\end{align*}
These expressions sum to $v_i-m_i=(T_Iv)_i$.
The untouched coordinates are additive as well.
Hence
\[
 u=u^-+u^+.
\]
Since $v^-$ and $v^+$ have upper bounds $h$ and $1-h$, respectively, the scaled form of Lemma~\ref{lem:transfer} gives
\[
 u^-\in hB,  \qquad  u^+\in(1-h)B.
\]

Thus
\[
 u  = h\,\frac{u^-}{h} +(1-h)\,\frac{u^+}{1-h}
 \]
is a proper convex combination of two points of $B$.
If $u\in E$, where $E$ is a face of $B$, both points lie in $E$.
This proves \eqref{eq:same-face}.
\end{proof}

For a nonempty polytope $A$, write $\dir(A)=\operatorname{span}\{a-b:a,b\in A\}$.

\begin{proposition}[Relative transversality]\label{prop:transversality}
The subdivisions $\mathcal S$ and $\mathcal T$ determined by $L$ and $U$ are transverse relative to every face of $B$.
\end{proposition}

\begin{proof}
Fix a face $E$ of $B$ and $x\in\relint(E)$, and let $A$ and $C$ be the cells of $\mathcal S|_E$ and $\mathcal T|_E$ containing $x$ in their relative interiors.
If $L(x)=U(x)$, then $L$ and $U$ are affine on $E$ by the proof of Lemma~\ref{lem:faceidentity}.
Hence $A=C=E$, and there is nothing to prove.

Assume $L(x)<U(x)$.
We first use the clipping decomposition to separate the directions coming from coordinates below and above $q$.
We then apply this separation to the conormal spaces of $A$ and $C$.

\smallskip
\noindent\emph{Step 1: separation of the two sets of coordinates.}
Choose
\[
 L(x)<h<U(x).
\]
After restricting to a neighbourhood of $x$ in $E$, we may assume $L(y)<h<U(y)$.
Let $v=T_I^{-1}y$.
For $p<q$,
\[
 v_p\leq L(y)<h,
\]
whereas for $r>q$,
\[
 v_r=y_r\geq U(y)>h.
\]
Thus, with
\[
 a=y^- -x^-,
\]
Lemma~\ref{lem:clipping} gives
\[
 a\in\dir(E),\qquad  a_p=y_p-x_p\quad(p<q),\qquad  a_r=0\quad(r>q).
\]

Put $V=\dir(E)$.
We claim that the restrictions to $V$ of linear forms supported below $q$ and those supported above $q$ have zero intersection.
Indeed, suppose that $\ell$ is supported below $q$, $m$ is supported above $q$, and
\[
 \ell|_V=m|_V.
\]
For $\delta\in V$, take $y=x+\varepsilon\delta$ with $\varepsilon>0$ sufficiently small.
The vector $a$ above then satisfies
\[
 \ell(a)=\varepsilon\ell(\delta),  \qquad
 m(a)=0.
\]
Since $a\in V$, we have $\ell(a)=m(a)$, and hence $\ell(\delta)=0$.
Thus the common restriction is zero.

\smallskip
\noindent\emph{Step 2: the cells meet transversely.}
Write $L=\max_i\ell_i$ as in Section~\ref{sec4}, and $U=\min_j g_j$, where the branches $g_j$ are the coordinate functions $u\mapsto u_r$ with $r>q$ together with the constant function $1$.
Let
\[
 S_L=\{i:\ell_i(x)=L(x)\},  \qquad  S_U=\{j:g_j(x)=U(x)\}
\]
be the sets of branches active at $x$.
Since $\ell_i(x)<L(x)$ for $i\notin S_L$ and $g_j(x)>U(x)$ for $j\notin S_U$, the remaining branches stay inactive in a neighbourhood of $x$.
Hence, near $x$,
\[
 A=E\cap\bigcap_{i,i'\in S_L}\{\ell_i=\ell_{i'}\},  \qquad  C=E\cap\bigcap_{j,j'\in S_U}\{g_j=g_{j'}\},
\]
and therefore, with $V=\dir(E)$ as above,
\[
 \dir(A)=V\cap\bigcap_{i,i'\in S_L}\ker(\ell_i-\ell_{i'}),  \qquad  \dir(C)=V\cap\bigcap_{j,j'\in S_U}\ker\bigl((g_j-g_{j'})_{\mathrm{lin}}\bigr),
\]
where $(\,\cdot\,)_{\mathrm{lin}}$ denotes the linear part.
Let $N_L=\dir(A)^\perp$ and $N_U=\dir(C)^\perp$ denote the annihilators of $\dir(A)$ and $\dir(C)$ in $V^*$.
These spaces are spanned by the restrictions to $V$ of the forms $\ell_i-\ell_{i'}$ and $(g_j-g_{j'})_{\mathrm{lin}}$, respectively.
The former are supported below $q$.
The linear parts of the differences $g_j-g_{j'}$ are supported above $q$.
Step~1 therefore gives
\[
 N_L\cap N_U=0.
\]
By definition, $\dir(A)=N_L^\perp$ and $\dir(C)=N_U^\perp$.
Since $x\in\relint(A)\cap\relint(C)$, the set $\relint(A)\cap\relint(C)$ is a relatively open neighbourhood of $x$ in $\operatorname{aff}(A)\cap\operatorname{aff}(C)$, so
\[
 \dir(A\cap C)=\dir(A)\cap\dir(C)=(N_L+N_U)^\perp.
\]
Writing $e=\dim E$, we obtain
\begin{align*}
 \dim(A\cap C)  &=e-\dim(N_L+N_U)\\
 &=e-\dim N_L-\dim N_U\\
 &=\dim A+\dim C-e.
\end{align*}
This is \eqref{eq:transverse}.
\end{proof}

The dimension formula in Definition~\ref{def:relative-transverse} is the polyhedral transversality condition.
A related formulation for normal fans of Minkowski sums appears in \cite[Sections~2--3]{FW10}.
Here it is proved directly for the subdivisions determined by $L$ and $U$,  after restriction to every face of the base.

\section{Proof of the face-number inequality}\label{sec6}

\begin{proposition}\label{prop:step}
If $I$ is an order ideal of $P$ and $q$ is minimal in $P\setminus I$, then
\[
 F_{\OO_{I,P\setminus I}(P)}(t)  \leq  F_{\OO_{I\cup\{q\},P\setminus(I\cup\{q\})}(P)}(t)
\]
coefficientwise.
\end{proposition}

\begin{proof}
By Proposition~\ref{prop:transversality}, the subdivisions $\mathcal S$ and $\mathcal T$ are transverse relative to the faces of $B$.
Lemma~\ref{lem:subdivision} and Lemma~\ref{lem:faceidentity} give
\[
 F_{\OO_{I\cup\{q\},P\setminus(I\cup\{q\})}(P)}-F_{\OO_{I,P\setminus I}(P)}
 =  F_B+F_{\mathcal S\vee\mathcal T}  -F_{\mathcal S}-F_{\mathcal T}  \geq0
\]
coefficientwise.
\end{proof}

Choose a linear extension $q_1,\ldots,q_n$ of $P$ and set
\[
 I_s=\{q_1,\ldots,q_s\}.
\]
Each $I_s$ is an order ideal, and $q_{s+1}$ is minimal in its complement.
Proposition~\ref{prop:step} gives
\[
 F_{\OO(P)}  =  F_{\OO_{\varnothing,P}(P)}  \leq  F_{\OO_{I_1,P\setminus I_1}(P)}  
 \leq\cdots\leq  F_{\OO_{P,\varnothing}(P)}  = F_{\CC(P)}.
\]
This proves Theorem~\ref{thm:face-inequality}.
The top-dimensional coefficients are both $1$.

More generally, if $I\subseteq I'$ are order ideals, list the elements of $I'\setminus I$ in the order induced by a linear extension of $P$.
Adding them successively and applying Proposition~\ref{prop:step} at each step gives
\[
 F_{\OO_{I,P\setminus I}(P)}\leq F_{\OO_{I',P\setminus I'}(P)}
\]
coefficientwise.

\section{The local comparison on the five-element star}
\label{sec7}

Let $P=\{a,b,q,c,d\}$ with $a,b<q<c,d$, where $a,b$ are incomparable and $c,d$ are incomparable.
Set $I=\{a,b\}$ and $J=\{q,c,d\}$, and write
\[
 Q=\OO_{I,J}(P),\qquad Q'=\OO_{I\cup\{q\},J\setminus\{q\}}(P).
\]
We construct a dimension-preserving injection $\Phi$ from the nonempty faces of $Q$ to those of $Q'$, following the count in Section~\ref{sec4}.

The common base is
\[
 B=\{u:0\leq u_a,u_b\leq u_c,u_d\leq1\}.
\]
Put
\[
 L=\max\{u_a,u_b\},\qquad U=\min\{u_c,u_d\},\qquad H=U-L.
\]
The fibres of $Q$ and $Q'$ are $[L(u),U(u)]$ and $[0,H(u)]$, respectively.
The subdivisions $\mathcal S$ and $\mathcal T$ have maximal cells
\[
 \begin{aligned}
 A_a&=B\cap\{u_a\geq u_b\},& A_b&=B\cap\{u_b\geq u_a\},\\
 C_c&=B\cap\{u_c\leq u_d\},& C_d&=B\cap\{u_d\leq u_c\}.
 \end{aligned}
\]
Their interior walls are
\[
 W_L=A_a\cap A_b=B\cap\{u_a=u_b\},\qquad W_U=C_c\cap C_d=B\cap\{u_c=u_d\}.
\]
For cells $A\in\mathcal S$, $C\in\mathcal T$, and $R\in\mathcal S\vee\mathcal T$, write
\[
 \begin{aligned}
 F_A^-&=\{(u,L(u)):u\in A\},\\
 F_C^+&=\{(u,U(u)):u\in C\},\\
 \Gamma_R&=\{(u,H(u)):u\in R\}.
 \end{aligned}
\]
For a face $E$ of $B$, put $E_0=E\times\{0\}$.
The carrier of a face of $Q$ or $Q'$ is the smallest face of $B$ containing its projection.

For each face $E$ of $B$, write
\[
 \begin{aligned}
 \widehat E&=\{(u,z):u\in E,\ L(u)\leq z\leq U(u)\},\\
 \widehat E'&=\{(u,w):u\in E,\ 0\leq w\leq H(u)\}.
 \end{aligned}
\]
By the face classification in Section~\ref{sec4}, the faces not contained in either fibre graph are precisely the $\widehat E$ for which $H$ is not identically zero on $E$.
For these, set
\[
 \Phi(\widehat E)=\widehat E'.
\]
Both faces have dimension $\dim E+1$.
This includes $\Phi(Q)=Q'$.
It remains to assign the graph faces.

\smallskip
\noindent\emph{Graph faces with proper carrier.}
For every proper face $E$ of $B$, at least one of $L|_E$ and $U|_E$ is affine.
Indeed, the facets of $B$ are given by
\[
 u_a=0,\quad u_b=0,\quad u_c=1,\quad u_d=1,\quad u_i=u_j\quad(i\in\{a,b\},\ j\in\{c,d\}).
\]
On the first two facets, $L$ is affine; on the next two, $U$ is affine.
On a facet $u_i=u_j$, one has $L=U=u_i$.

Define
\[
 \psi_E(u,z)=
 \begin{cases}
  (u,z-L(u)),&\text{if $L|_E$ is affine},\\
  (u,U(u)-z),&\text{otherwise}.
 \end{cases}
\]
In the second case, $U|_E$ is affine.
Thus $\psi_E$ is an affine isomorphism from $\widehat E$ to $\widehat E'$.
For each graph face $F$ with carrier $E$, set
\[
 \Phi(F)=\psi_E(F).
\]
Since $\widehat E$ and $\widehat E'$ are faces of $Q$ and $Q'$, this sends faces to faces.
It gives a bijection between their graph faces with carrier $E$ and preserves dimensions.
If $L=U$ on $E$, the degenerate lifted face is sent to $E_0$.

For example, on the square $E=B\cap\{u_c=u_d=1\}$, the map is $(u,z)\mapsto(u,1-z)$.
It sends the old upper square to $E_0$, the two old lower triangles to $\Gamma_{E\cap A_a}$ and $\Gamma_{E\cap A_b}$, and their common edge to $\Gamma_{E\cap W_L}$.

\smallskip
\noindent\emph{Graph faces with carrier $B$.}
The old graph faces with carrier $B$ are the four facets
\[
 F^-_{A_a},\quad F^-_{A_b},\quad F^+_{C_c},\quad F^+_{C_d}
\]
and the two three-dimensional faces $F^-_{W_L}$ and $F^+_{W_U}$.

For the latter two, the refinement cells of the same dimension give the candidate sets
\[
 \{\Gamma_{W_L\cap C_c},\Gamma_{W_L\cap C_d}\},\qquad \{\Gamma_{W_U\cap A_a},\Gamma_{W_U\cap A_b}\}.
\]
These sets are disjoint: a common three-dimensional face would project into $W_L\cap W_U$, which has dimension two.
This is the disjointness used for $k<\dim E$ in Section~\ref{sec4}.
Choose
\[
 \Phi(F^-_{W_L})=\Gamma_{W_L\cap C_c},\qquad \Phi(F^+_{W_U})=\Gamma_{W_U\cap A_a}.
\]

For the facets, the incidence graph from Section~\ref{sec4} is $K_{2,2}$, with vertex classes $\{A_a,A_b\}$ and $\{C_c,C_d\}$.
Its four edges correspond to the cells $A_i\cap C_j$.
Choose the spanning tree with edges
\[
 A_aC_c,\qquad A_aC_d,\qquad A_bC_c,
\]
rooted at $A_a$.
Assign the root to $B_0$ and each other vertex to the new upper facet corresponding to its edge towards the root.
Thus
\[
 \begin{aligned}
 \Phi(F^-_{A_a})&=B_0,& \Phi(F^-_{A_b})&=\Gamma_{A_b\cap C_c},\\
 \Phi(F^+_{C_c})&=\Gamma_{A_a\cap C_c},& \Phi(F^+_{C_d})&=\Gamma_{A_a\cap C_d}.
 \end{aligned}
\]
The targets are distinct.
Here the bottom face supplies the additional term needed in the case $k=\dim E$.

These rules define $\Phi$ on every nonempty face of $Q$.
They preserve dimension and carrier.
For each proper carrier the graph assignment is bijective, and the six graph assignments with carrier $B$ are injective.
The full inverse-image faces cannot collide with graph faces, since they contain points in the interiors of nondegenerate fibres.
Hence $\Phi$ is injective.

The faces of $Q'$ outside its image are precisely
\[
 \begin{array}{c|c}
  \text{face}&\text{dimension}\\ \hline
  \Gamma_{W_L\cap W_U}&2\\
  \Gamma_{W_L\cap C_d},\quad\Gamma_{W_U\cap A_b}&3\\
  \Gamma_{A_b\cap C_d}&4.
 \end{array}
\]
Consequently,
\[
 F_{Q'}(t)-F_Q(t)=t^2+2t^3+t^4.
\]

\section{Marked monotonicity and strict inequalities}\label{sec:extensions}

\subsection{Admissible marked chain--order polytopes}
A \emph{marked poset} $(P,A,\lambda)$ consists of a finite poset $P$, a subset $A$ containing all minimal and maximal elements, and an order-preserving marking $\lambda:A\to\mathbb R$.
Put $V=P\setminus A$ and extend $x\in\mathbb R^V$ by $x_a=\lambda(a)$ for $a\in A$.
The marked order and chain polytopes of \cite{ABS11} are
\begin{align*}
 \OO(P,\lambda)
 &=\{x\in\mathbb R^V:x_p\leq x_q\text{ whenever }p<q\text{ in }P\},\\
 \CC(P,\lambda)
 &=\left\{x\in\mathbb R_{\geq0}^V:
       \sum_{i=1}^s x_{p_i}\leq\lambda(b)-\lambda(a)
       \text{ for every marked chain}\right\},
\end{align*}
where a marked chain is $a<p_1<\cdots<p_s<b$ with $a,b\in A$ and $p_i\in V$.

A partition $V=D\sqcup(V\setminus D)$ is \emph{admissible} if $D$ is an order ideal of $V$ \cite{FF16}.
Its marked chain--order polytope $Q_D=Q_D(P,\lambda)$ consists of the $x\in\mathbb R^V$ satisfying $x_d\geq0$ for $d\in D$ and
\[
 x_a+\sum_{i=1}^s x_{p_i}\leq x_b
 \quad(a<p_1<\cdots<p_s<b,\quad p_i\in D,\quad
 a,b\in A\cup(V\setminus D)).
\]
The case $s=0$ is included.
Thus $Q_\varnothing=\OO(P,\lambda)$ and $Q_V=\CC(P,\lambda)$.

\begin{theorem}\label{thm:marked-monotonicity}
For order ideals $D\subseteq D'\subseteq V$, one has $F_{Q_D}\leq F_{Q_{D'}}$ coefficientwise.
\end{theorem}
\begin{proof}
Choose $q\in\min(V\setminus D)$.
All unmarked elements below $q$ belong to $D$, and none above $q$ does.
Let $L$ be the maximum of $\lambda(a)+\sum_i x_{p_i}$ over chains $a<p_1<\cdots<p_s<q$ with $a\in A$, $p_i\in D$, and let $U=\min_{q\prec b}x_b$, using the marking when $b\in A$.
Here $\prec$ denotes a cover.
As in Proposition~\ref{prop:fibers}, the two polytopes have a common projection $B=B_\lambda$ and fibres
\begin{align*}
 Q_D&=\{(u,z):u\in B,\ L(u)\leq z\leq U(u)\},\\
 Q_{D\cup\{q\}}&=\{(u,w):u\in B,\ 0\leq w\leq U(u)-L(u)\}.
\end{align*}
The proof of Lemma~\ref{lem:faceidentity} applies to these fibres.
We verify relative transversality by marked clipping.

Retaining marked coordinates, the partial transfer is
\[
 (T_Dv)_p=
 \begin{cases}
 v_p-\max_{r\prec p}v_r,&p\in D,\\
 v_p,&p\notin D;
 \end{cases}
\]
see \cite[Theorem~2.1]{Fan+20}.
For $h\in\mathbb R$, put $v^-_p=\min(v_p,h)$, $v^+_p=(v_p-h)_+$, and define $\lambda^\pm$ likewise.
Clipping commutes with predecessor maxima, so $T_Dv=T_Dv^-+T_Dv^+$.
Adding the defining inequalities gives the reverse inclusion, hence
\[
 Q_D(P,\lambda)=Q_D(P,\lambda^-)+Q_D(P,\lambda^+),\qquad
 B_\lambda=B_{\lambda^-}+B_{\lambda^+}.
\]
This is the clipping case of \cite[Theorem~20]{FFP20}.

Fix $x\in\relint(E)$ for a face $E$ of $B$.
If $L(x)=U(x)$, nonnegative concavity forces $L=U$ to be affine on $E$, so both restricted subdivisions are trivial.
Otherwise choose $L(x)<h<U(x)$.
For nearby $y\in E$, lift $(y,h)$ by $T_D^{-1}$, clip, and project away $q$, obtaining $y=y^-+y^+$.
A functional exposing $E$ exposes summand faces $E^-,E^+$ with $E=E^-+E^+$; equality of support values puts $y^\pm\in E^\pm$.
Since $E^-+x^+\subseteq E$, we have $a=y^--x^-\in\dir(E)$.
Below $q$, clipping preserves the transferred coordinates; above $q$, all unmarked coordinates are order coordinates and clip to $h$.
Thus
\[
 a_p=y_p-x_p\quad(p<q),\qquad a_p=0\quad(p>q).
\]
If lower- and upper-supported forms $\ell,g$ agree on $\dir(E)$, putting $y=x+\varepsilon d$ gives $\varepsilon\ell(d)=\ell(a)=g(a)=0$.
Their restricted spaces have zero intersection.
The linear parts of differences between active branches span the two conormal spaces, so the dimension argument of Proposition~\ref{prop:transversality} applies.
Lemma~\ref{lem:subdivision} proves the one-step inequality.
Add $D'\setminus D$ along a linear extension to finish.
The argument uses relative dimensions and allows degenerate markings.
\end{proof}

\subsection{The equality case of the Hibi--Li conjecture}
The next theorem proves the second part of the Hibi--Li conjecture \cite[Conjecture~2.4(b)]{HL16} for $2\leq k\leq n-1$, and extends it to marked polytopes.
The restriction $k\geq2$ is necessary: ordinary order and chain polytopes always have the same numbers of vertices and edges \cite{Sta86,Hib+17}.

\begin{theorem}\label{thm:marked-strictness}
Let the marked order and chain polytopes of $(P,A,\lambda)$ have dimension $n\geq3$.
Either they are affinely equivalent, or
\[
 f_k(\OO(P,\lambda))<f_k(\CC(P,\lambda))
 \qquad(2\leq k\leq n-1).
\]
For integral markings the equivalence can be chosen unimodular, with respect to the lattices in the affine hulls.
\end{theorem}
\begin{proof}
By \cite[Propositions~6.1--6.2 and Remark~6.3]{Fan+20}, we may regularize the endpoint polytopes, preserving their affine types and, for integral markings, their relative lattice types.
We retain the notation $P,A,\lambda,V$ for the reduced marked poset.
Regularity means that for every cover $p\prec r$ and marked $a\leq r$, $b\geq p$, either $a=b$ or $\lambda(a)<\lambda(b)$.
Now $n=|V|$, and boundedness ensures that all extrema remain marked.

At an admissible step $q$, the function $L$ is affine if there is only one saturated path from a marked element to $q$ with no internal marks; $U$ is affine if $q$ has only one upper cover.
Then $(u,z)\mapsto(u,z-L(u))$ or $(u,z)\mapsto(u,U(u)-z)$, respectively, is an affine equivalence, unimodular for integral markings.
If every $q\in V$ satisfies one of these conditions, composition proves the equivalence assertion.
Otherwise choose $q$ with at least two incoming paths and two upper covers, and set
\[
 D=V\setminus\uparrow q,\qquad \uparrow q=\{p\in P:q\leq p\}.
\]
Transfer the order ideal $D$ first, then $q$, then the remaining elements.
All step differences are nonnegative by Theorem~\ref{thm:marked-monotonicity}.
We produce the strict inequalities at the step $q$.

\smallskip
\noindent\emph{A small face.}
For $p\in V$, put
\[
 \ell_p=\max_{\substack{a\in A\\a<p}}\lambda(a),\qquad
 r_p=\min_{\substack{b\in A\\p<b}}\lambda(b),\qquad
 \alpha=\ell_q<\beta=r_q.
\]
Follow unique lower covers backwards from $q$ to the first branching element $p$.
It occurs before a mark, and $\ell_p=\alpha$.
Every unmarked $i\prec p$ satisfies $\ell_i\leq\alpha<r_i$ by regularity.
If some $\ell_i<\alpha$, select that $i$.
Otherwise select two unmarked predecessors $i_1,i_2$ with lower bound $\alpha$.
They exist: two marked predecessors contradict regularity, as does a marked predecessor $a$ beside an unmarked $i$ with $\ell_i=\alpha$, since then $\alpha<\lambda(a)\leq\alpha$.

Set all other chain coordinates to zero.
The lower envelopes and crossing coordinates are
\[
 \begin{array}{c|c}
 L&\text{crossing coordinates}\\ \hline
 \max(\alpha,\ell_i+x_i)&x_i=\alpha-\ell_i\\
 \alpha+\max(x_{i_1},x_{i_2})&x_{i_1}=x_{i_2}=\varepsilon.
 \end{array}
\]
Here $\varepsilon>0$ is sufficiently small.
Paths attaining the lower markings give these branches: in the first case a path starting at value $\alpha$ avoids $i$; in the second no path contains both witnesses.

Write $y$ for order coordinates.
There is at most one marked upper cover $b$ of $q$.
If it exists, select another upper cover $j$.
Regularity gives $\ell_j<\beta=\lambda(b)<r_j$; take $y_j=\beta$, so locally $U=\min(\beta,y_j)$.
Otherwise select two covers $j_1,j_2$ and take $y_{j_1}=y_{j_2}=\beta-\varepsilon$; then locally $U=\min(y_{j_1},y_{j_2})$, since $\ell_{j_t}<\beta\leq r_{j_t}$.
Fix every other order coordinate $y_p$ to $r_p$.
These supporting equalities and the zero chain coordinates define a face $E$ of $B$ with $m\in\{2,3,4\}$ selected coordinates.

Every base inequality involving a selected coordinate is strict at this point.
A chain containing a selected $i$ has incoming value at most $\ell_i+x_i\leq\alpha+\varepsilon<r_i$ and below every selected upper value; a fixed endpoint above $i$ has value at least $r_i$.
A chain without selected lower coordinates ending at a selected $j$ has incoming value at most $\ell_j<y_j$.
There are no unmarked order coordinates below $j$ in $B$, since $j$ covers $q$; above $j$, we have $y_j<r_j\leq r_p$.
Thus $\dim E=m$, the point lies in $\relint(E)$, and $L<U$.
The two walls cross transversely in the selected coordinates.

\smallskip
\noindent\emph{Higher dimensions.}
The crossing persists in every face $F\supseteq E$.
All base inequalities active at the chosen point involve only coordinates fixed on $E$, so locally $F$ is a product with an open set in the selected coordinates.
Move the other coordinates slightly into the relative interior of their factor.
Group lower branches by their selected terms: $0,x_i$, or $x_{i_1},x_{i_2}$, respectively; in the latter case the zero group remains inactive.
Within each group, differences depend only on the other coordinates.
Choose the other coordinates off the tie hyperplanes within each group that are proper in the affine hull of their factor, identifying branches with identical restrictions.
The maximum within each group is then affine locally.
Adjust one selected lower coordinate to restore the tie between these two maxima.
The upper tie is unchanged, and all other upper branches remain inactive.
Hence two walls cross in $\relint(F)$, still with $L<U$.

Put $e=\dim F$.
Their intersection gives an $(e-2)$-cell of the refinement that is contained in no original $(e-2)$-cell.
Each wall splits into at least two $(e-1)$-cells, giving two excess cells in the disjoint candidate sets of Lemma~\ref{lem:subdivision}.
The four adjacent full-dimensional cells give a four-cycle in its connected overlap graph.
The contribution with carrier $F$ is therefore at least $t^{e-2}(1+t)^2$, coefficientwise.
Along a saturated flag from $E$ to $B$, the dimensions $e$ run from $m$ to $\dim B=n-1$.
Since $m\leq4$, these contributions yield positive coefficients in every degree $2,\ldots,n-1$.
The other step differences are nonnegative, completing the proof.
\end{proof}

For an ordinary $n$-element poset with $n\geq3$, failure of unimodular equivalence is therefore equivalent to strict inequality in one, and hence every, dimension $2,\ldots,n-1$.

\printbibliography
\end{document}